\documentclass[11pt]{article}
\usepackage[margin=0.85in]{geometry}
\usepackage{amsmath,amssymb,amsthm,mathtools,microtype}
\usepackage[utf8]{inputenc}
\usepackage[T1]{fontenc}
\usepackage{lmodern}
\usepackage[hidelinks]{hyperref}
\newtheorem{theorem}{Theorem}[section]
\newtheorem{proposition}[theorem]{Proposition}

\newtheorem{corollary}[theorem]{Corollary}
\theoremstyle{remark}
\newtheorem{remark}[theorem]{Remark}
\newcommand{\R}{\mathbb R}
\newcommand{\Z}{\mathbb Z}
\newcommand{\C}{\mathbb C}
\newcommand{\ellp}{\ell^2(\Z_+)}
\newcommand{\ellZ}{\ell^2(\Z)}
\newcommand{\cH}{\mathcal H}
\newcommand{\cU}{\mathcal U}
\newcommand{\sL}{\mathsf L}
\newcommand{\sP}{\mathsf P}
\newcommand{\ind}{\mathbf1}
\newcommand{\dual}{\mathrm{dual}}
\newcommand{\ASC}{\mathrm{ASC}}
\renewcommand{\Im}{\operatorname{Im}}
\newcommand{\diag}{\operatorname{diag}}
\newcommand{\sgn}{\operatorname{sgn}}
\newcommand{\Span}{\operatorname{span}}

\hypersetup{pdftitle={The Fixed-q Edge Limit of the q-Krawtchouk Ensemble},pdfauthor={Pavel Nikitin}}
\title{The Fixed-$q$ Edge Limit of the $q$-Krawtchouk Ensemble}
\author{Pavel Nikitin\thanks{Email: \texttt{pnikitin@bimsa.cn}, \texttt{pnikitin0103@yahoo.co.uk}}\\[4pt]
\small Beijing Institute of Mathematical Sciences and Applications (BIMSA),\\
\small Beijing 101408, People's Republic of China}
\date{}

\begin{document}
\maketitle

\begin{abstract}
We study the fixed-$q$ right-edge asymptotics of the $q$-Krawtchouk ensemble. For $0<q<1$, finite polynomial duality gives a direct proof of convergence of the correlation kernels to an explicit limiting kernel. The local Jacobi limit is a shifted continuous $q^{-1}$-Hermite expression with an indeterminate moment problem, so the limiting Jacobi expression does not determine a unique self-adjoint operator. A single spectral atom, asymptotically saturating the mass bound, selects the realization and yields strong-resolvent convergence. The limiting kernel represents its positive spectral projection and admits spectral-series, Christoffel--Darboux, and partial-theta representations. After a ground-state transformation, the associated dual operator agrees with a two-step Markov operator of Borodin--Corwin. For $q>1$, the limiting Jacobi operator is of Al-Salam--Carlitz~I type and is essentially self-adjoint, so kernel convergence follows from the standard Borodin--Olshanski spectral-projection method.
\end{abstract}

\section{Introduction}\label{sec:introduction}

The $q$-Krawtchouk ensemble is an orthogonal polynomial ensemble on $\{0,\ldots,M\}$, with $N$ particles, $0\le N\le M$, and parameters $p>0$ and $q>0$, $q\ne1$. Its weight and correlation kernel are defined in Section~\ref{subsec:finite-ensemble}.

In the standard scaling where $q\to1$ with $1-q$ of order $M^{-1}$, known local limits include the discrete sine kernel in the bulk, the Airy kernel at regular edges, and the discrete Hermite kernel in critical boundary regimes; see Betea--Nazarov--Nikitin--Scrimshaw~\cite{BNNS}. For a constant-$q$ specialization arising from skew Howe duality, the same authors conjectured that the corner asymptotics are governed by a $q$-analogue of the discrete Hermite kernel.

In this paper we study the right-edge asymptotics with $q$ fixed, allowing $N=N_M$ and $p=p_M$ to depend on $M$. Our main regime is $0<q<1$, with both $N$ and $M-N$ tending to infinity and
\[
p_M q^{2N} \longrightarrow \rho\in(0,\infty).
\]
We determine the limiting correlation kernel, identify its self-adjoint spectral realization, and derive several explicit formulas.

A useful feature of the $q$-Krawtchouk ensemble is its dual representation, which interchanges the roles of the spatial variable and the polynomial degree. In this representation, the asymptotics of the polynomials and their weights give a direct proof of convergence of the correlation kernels. We also obtain Christoffel--Darboux and partial-theta representations of the limiting kernel and relate the limiting dual operator to a two-step Markov operator of Borodin--Corwin~\cite{BC}.

Identifying the self-adjoint operator associated with this limit is a separate problem. The spectral-projection method of Borodin--Olshanski~\cite{BO2007,BO} obtains limiting correlation kernels from convergence of the associated Jacobi operators. When the limiting Jacobi expression is essentially self-adjoint, it determines a unique self-adjoint operator.

In our fixed-$q$ regime, the limiting expression is of shifted continuous $q^{-1}$-Hermite type and has an indeterminate moment problem, so it admits more than one self-adjoint realization. We identify the realization selected by the finite ensemble and show that its positive spectral projection gives the limiting kernel.

This realization can also be recovered from the limiting mass of a single spectral atom. We develop this observation into a general maximal-atom selection principle and use it to prove strong-resolvent convergence of the finite Jacobi operators.

Direct asymptotic analysis has been used to obtain limiting correlation kernels; see, for example, Gorin--Olshanski~\cite{GO} and Cuenca--Gorin--Olshanski~\cite{CGO}. Bornemann~\cite{Bornemann} studied related kernel limits, including hard-edge problems where a prescribed boundary condition selects the self-adjoint realization. Our maximal-atom selection principle instead determines the realization from a single limiting spectral mass, without prescribing a boundary condition; this selection mechanism appears to be new. A related issue arises in the fixed-$q$ $q$-Racah analysis of Knizel--Petrov~\cite{KP}, where the limiting Jacobi expression admits more than one self-adjoint realization and the corresponding selection problem is left open.

For comparison, we obtain a companion limit for $q>1$ in the critical window $M-2N=O(1)$. The limiting Jacobi operator is of Al-Salam--Carlitz~I type and is essentially self-adjoint, so kernel convergence follows from the standard Borodin--Olshanski spectral-projection method. We also give explicit formulas for the corresponding limiting kernel.

\subsection{Main results}\label{subsec:main-results}

Write $\Z_+ := \{0,1,2,\ldots\}$, and let $(e_n)_{n\in\Z_+}$ denote the canonical basis of $\ellp$. Assume
\[
0<q<1, \qquad N = N_M \to \infty, \qquad M-N_M \to \infty, \qquad p_M q^{2N_M} \to \rho\in(0,\infty).
\]
Let $A_M$ be the reflected $q$-Krawtchouk Jacobi matrix, scaled by $p_M^{-1/2}$ and centered at the degree-$N_M$ eigenvalue as in Section~\ref{sec:q-less-one-scaling}, and let $\widehat A_M = A_M \oplus 0$ on $\ellp$. Its coefficients converge, for each fixed $n\in\Z_+$, to
\[
A_{\rho,\mathrm{formal}} = J+c_\rho I, \qquad
(Jf)_n = \sqrt{q^{-n-1}-1}\,f_{n+1} + \sqrt{q^{-n}-1}\,f_{n-1}, \qquad
c_\rho = \rho^{-1/2}-\rho^{1/2},
\]
with the usual convention that the $f_{-1}$ term is absent.

Set
\[
s_l = \rho^{1/2}q^{-l} - \rho^{-1/2}q^l, \qquad l\in\Z,
\]
\[
\mu_l = \frac{\rho^{-2l}(1+\rho^{-1}q^{2l})q^{l(2l-1)}}
{(-\rho^{-1},-q\rho,q;q)_\infty}, \qquad P_n(s) = \frac{q^{n(n+1)/4}}{\sqrt{(q;q)_n}}h_n(s/2\mid q).
\]
Here $h_n(x\mid q)$ denotes the continuous $q^{-1}$-Hermite polynomial of degree $n$.
Let $\widehat K_{M,N}$ denote the orthogonal projection defined by the reflected correlation kernel, extended by zero to $\ellp$.

\begin{theorem}[Fixed-$q$ $q$-Krawtchouk edge limit, $0<q<1$]\label{thm:intro-main-qsmall}
There is a limiting orthogonal projection $K_\rho$ with matrix kernel
\[
K_\rho(m,n) = \sum_{l=1}^{\infty}\mu_l P_m(s_l)P_n(s_l), \qquad m,n\in\Z_+,
\]
such that
\[
\widehat K_{M,N} \longrightarrow K_\rho
\]
strongly on $\ellp$. In particular, $K_{M,N}(m,n) \longrightarrow K_\rho(m,n)$ for every fixed $m,n\in\Z_+$.

The spectral data
\[
\lambda_l = s_l+c_\rho, \qquad \mu_l, \qquad l\in\Z,
\]
define a distinguished self-adjoint realization $A_\rho$ of $A_{\rho,\mathrm{formal}}$, and
\[
K_\rho = \ind_{(0,\infty)}(A_\rho).
\]
Moreover,
\[
\widehat A_M \xrightarrow{\mathrm{s.r.}} A_\rho.
\]

The kernel $K_\rho$ also admits Christoffel--Darboux and partial-theta representations, developed in Section~\ref{sec:explicit-kernel}.
\end{theorem}

The strong-resolvent statement is an instance of the general selection principle proved in Section~\ref{sec:selection-principle}: under coefficientwise convergence to an indeterminate Jacobi problem, convergence of a distinguished moving spectral atom to the maximal admissible mass at its limiting point selects the unique $N$-extremal measure containing that point and hence the corresponding self-adjoint extension.

The opposite fixed-$q$ regime gives a companion result in which the extension-selection step disappears. Assume
\[
q>1, \qquad M \to \infty, \qquad p_M q^{2N_M} \to \rho\in(0,\infty),
\]
and, for some $\kappa\in\Z$,
\[
M-2N_M = \kappa
\]
for all sufficiently large $M$. Set
\[
Q := q^{-1}\in(0,1), \qquad c := (\rho q^\kappa)^{-1}.
\]
Let $L_M$ be the reflected $q$-Krawtchouk Jacobi matrix in the $q>1$ convention, centered by taking the degree-$N_M$ eigenvalue minus the reflected matrix and scaled by $(p_M q^M)^{-1}$. Introduce the fixed alternating gauge
\[
Ge_n := (-1)^ne_n, \qquad L_M^+ := G_M L_M G_M.
\]
Here $G_M$ denotes the restriction of $G$ to $\Span\{e_0,\ldots,e_M\}$.
The Jacobi coefficients of $L_M^+$ converge to
\[
a_n^{\ASC} = \sqrt{cQ^n(1-Q^{n+1})}, \qquad b_n^{\ASC} = (1-c)Q^n.
\]
Let $L$ be the unique self-adjoint Al-Salam--Carlitz I Jacobi operator with these coefficients.

\begin{theorem}[Critical $q>1$ limit]\label{thm:intro-main-qgreater}
With $\widehat L_M = L_M \oplus 0$ and $\widehat L_M^+ = L_M^+ \oplus 0$,
\[
\widehat L_M^+ \xrightarrow{\mathrm{s.r.}} L, \qquad \widehat L_M \xrightarrow{\mathrm{s.r.}} GLG.
\]
Moreover,
\[
G\widehat K_{M,N}G \longrightarrow \ind_{(-\infty,0)}(L)
\]
strongly on $\ellp$. Writing $K_{\ASC}$ for the matrix kernel of $\ind_{(-\infty,0)}(L)$, for every fixed $m,n\in\Z_+$,
\[
K_{M,N}(m,n) \longrightarrow (-1)^{m+n}K_{\ASC}(m,n).
\]
The limiting kernel admits explicit spectral-series and integrable representations in terms of Al-Salam--Carlitz I polynomials; these are developed in Section~\ref{sec:qgreater1-main}.
\end{theorem}

The paper is organized as follows. Section~\ref{sec:finite-spectral} introduces the finite $q$-Krawtchouk ensemble, its Jacobi representation, and the finite dual transform. Section~\ref{sec:q-less-one-scaling} derives the $0<q<1$ right-edge Jacobi limit. Section~\ref{sec:kernel-convergence} proves convergence of the correlation kernels directly from the finite dual representation. Section~\ref{sec:selected-extension} identifies the corresponding self-adjoint realization of the limiting Jacobi expression, and Section~\ref{sec:selection-principle} isolates the maximal-atom selection mechanism and proves strong-resolvent convergence. Section~\ref{sec:explicit-kernel} develops the limiting dual operator, its relation to Borodin--Corwin, and the Christoffel--Darboux and partial-theta representations of the kernel. Finally, Section~\ref{sec:qgreater1-main} treats the companion $q>1$ regime.

\section{The \texorpdfstring{$q$}{q}-Krawtchouk ensemble and its spectral formulation}\label{sec:finite-spectral}

\subsection{The finite orthogonal polynomial ensemble}\label{subsec:finite-ensemble}

For the main regime let $0<q<1$, let $M\in\Z_{\ge0}$ and $p=p_M>0$, and consider the ordinary $q$-Krawtchouk polynomials in the standard convention~\cite[Eqs.~(14.15.1)--(14.15.2)]{KLS}:
\[
K_k(q^{-x};p,M;q) =
{}_3 \phi_2\!\left(\begin{matrix}q^{-k},q^{-x},-pq^k\\ q^{-M},0\end{matrix};q,q\right), \qquad
0\le k,x\le M.
\]
Their orthogonality weight on $\{0,1,\ldots,M\}$ is
\[
w_x = \frac{(q^{-M};q)_x}{(q;q)_x}(-p)^{-x}, \qquad x = 0,\ldots,M.
\]
Let $h_k$ be the squared norm of $K_k$ in this orthogonality relation and put
\[
\phi_{k,M}(x) := \frac{\sqrt{w_x}\,K_k(q^{-x};p,M;q)}{\sqrt{h_k}}, \qquad k,x = 0,\ldots,M.
\]
Then $\{\phi_{k,M}\}_{k=0}^M$ is an orthonormal basis of $\ell^2(\{0,\ldots,M\})$.  For $N\in\{0,\ldots,M\}$, the associated $N$-particle orthogonal polynomial ensemble has correlation kernel
\begin{equation}\label{eq:finite-kernel-projection-sum}
K_{M,N}(x,y) = \sum_{k=0}^{N-1}\phi_{k,M}(x)\phi_{k,M}(y).
\end{equation}
The corresponding operator is the orthogonal projection onto $\Span\{\phi_{0,M},\ldots,\phi_{N-1,M}\}$.

\subsection{The finite Jacobi operator}

The $q$-Krawtchouk polynomials satisfy the $q$-difference equation~\cite[Eq.~(14.15.5)]{KLS}
\begin{equation}\label{eq:qk-finite-difference}
\Lambda_k y(x) = (1-q^{x-M})y(x+1) - \bigl[(1-q^{x-M}) - p(1-q^x)\bigr]y(x) - p(1-q^x)y(x-1),
\end{equation}
where $y(x) = K_k(q^{-x};p,M;q)$ and
\begin{equation}\label{eq:qk-finite-eigenvalue}
\Lambda_k = q^{-k}(1-q^k)(1+pq^k) = q^{-k}-1+p-pq^k.
\end{equation}
The weight ratio $\frac{w_{x+1}}{w_x} = \frac{1-q^{x-M}}{-p(1-q^{x+1})}$ shows that the operator in \eqref{eq:qk-finite-difference} is symmetric in the weighted space $\ell^2(\{0,\ldots,M\},w)$.  Conjugating by multiplication by $w_x^{1/2}$ transfers it to the standard Euclidean space $\ell^2(\{0,\ldots,M\})$.  We then reflect the right edge by the change of coordinates
\[
n = M-x.
\]
We denote by $H_{p,M}$ the resulting operator on $\ell^2(\{0,\ldots,M\})$.  It is the real symmetric tridiagonal matrix with coefficients
\begin{equation}\label{eq:finite-reflected-offdiag}
\widetilde a_n^{(M)} = -\sqrt{p(q^{-n-1}-1)(1-q^{M-n})}, \qquad 0\le n<M,
\end{equation}
\begin{equation}\label{eq:finite-reflected-diag}
\widetilde b_n^{(M)} = q^{-n}-1 + p(1-q^{M-n}), \qquad 0\le n\le M.
\end{equation}
The normalized degree-$k$ eigenvector of $H_{p,M}$ is the reflected function
\[
\psi_{k,M}(n) = \phi_{k,M}(M-n),
\]
and
\[
H_{p,M}\psi_{k,M} = \Lambda_k \psi_{k,M}, \qquad k = 0,\ldots,M.
\]
From this point on, we also use reflected coordinates for the kernel:
\[
K_{M,N}(m,n) := K_{M,N}(M-m,M-n),
\]
where the right-hand side denotes the kernel in the original
coordinates introduced in \eqref{eq:finite-kernel-projection-sum}.
Thus
\[
K_{M,N}(m,n) = \sum_{k=0}^{N-1}\psi_{k,M}(m)\psi_{k,M}(n).
\]
Viewed as an operator on $\ell^2(\{0,\ldots,M\})$, $K_{M,N}$ is therefore the orthogonal spectral projection of $H_{p,M}$ onto $\Span \{\psi_{0,M},\ldots,\psi_{N-1,M}\}$.

\subsection{The finite dual transform}\label{subsec:finite-dual-transform}

The projection $K_{M,N}$ also has a complementary description in the polynomial-degree variable.  Let
\[
\mathcal I_M := \{N-M,N-M+1,\ldots,N\},
\]
and reindex the polynomial degree by
\begin{equation}\label{eq:dual-index}
l = N-k.
\end{equation}
At the reflected endpoint $n=0$, the identity
\[
K_k(q^{-M};p,M;q) = (-pq^k)^k
\]
shows that $\sgn \psi_{N-l,M}(0) = (-1)^{N-l}$. For the later comparison with spectral weights, it is convenient to fix the remaining global sign so that the endpoint values are positive. We therefore define
\begin{equation}\label{eq:finite-dual-vectors}
v_{n,M}(l) := (-1)^{N+l}\psi_{N-l,M}(n), \qquad l\in\mathcal I_M,
\end{equation}
so that $v_{0,M}(l)>0$ for $l\in\mathcal I_M$. Since the matrix $(\psi_{k,M}(n))_{k,n=0}^M$ is orthogonal, the map
\begin{equation}\label{eq:finite-dual-transform}
\cU_M:\ell^2(\{0,\ldots,M\}) \longrightarrow \ell^2(\mathcal I_M), \qquad (\cU_M e_n)(l) = v_{n,M}(l),
\end{equation}
is unitary.

The $q$-Krawtchouk/dual-$q$-Krawtchouk duality is
\begin{equation}\label{eq:qk-duality-kls}
K_k(q^{-x};p,M;q) = \mathcal K_x(\lambda(k);-pq^M,M\mid q), \qquad \lambda(k) = q^{-k}-pq^k,
\end{equation}
where $\mathcal K$ denotes the dual $q$-Krawtchouk polynomial.  Thus \eqref{eq:finite-dual-transform} is the normalized dual $q$-Krawtchouk transform, followed by the alternating Jacobi gauge; see \cite[Remark at the end of \S14.17]{KLS}.

\begin{proposition}[Exact finite half-lattice projection]
\label{prop:finite-dual-projection}
For every finite $M,N$ in the $0<q<1$ regime,
\begin{equation}\label{eq:finite-dual-projection}
K_{M,N} = \cU_M^*\,\ind_{\{1,\ldots,N\}}\,\cU_M.
\end{equation}
\end{proposition}

\begin{proof}
Using \eqref{eq:dual-index} and \eqref{eq:finite-dual-vectors}, the finite projection sum becomes
\[
K_{M,N}(m,n) = \sum_{k=0}^{N-1}\psi_{k,M}(m)\psi_{k,M}(n) = \sum_{l=1}^{N}v_{m,M}(l)v_{n,M}(l).
\]
The alternating signs disappear because they occur twice.  This is \eqref{eq:finite-dual-projection}.
\end{proof}

\subsection{Common Hilbert space and embedding at the right edge}\label{subsec:common-hilbert-space}

All operator limits below are taken on the fixed Hilbert space
\[
\cH := \ellp.
\]
Let
\[
\cH_M := \ell^2(\{0,\ldots,M\}), \qquad \cH_M^\perp = \ell^2(\{M+1,M+2,\ldots\}),
\]
so that $\cH = \cH_M \oplus \cH_M^\perp$.  Whenever $T_M$ is a finite Jacobi matrix on $\cH_M$, we denote by
\[
\widehat T_M := T_M \oplus 0
\]
its zero extension to $\cH$.  Strong-resolvent convergence below refers to these extensions on $\cH$.

If $E_{T_M}$ and $E_{\widehat T_M}$ denote the corresponding spectral resolutions, then for every Borel set $B\subset\R$,
\[
E_{\widehat T_M}(B) = E_{T_M}(B) \oplus \ind_B(0)I_{\cH_M^\perp}.
\]
In particular, whenever $0\notin B$, $\ind_B(\widehat T_M) = \ind_B(T_M) \oplus 0$. The spectral projections defining the correlation kernels below have $B=(0,\infty)$ for $0<q<1$ and $B=(-\infty,0)$ for $q>1$.  Hence the added zero eigenspace does not alter the finite correlation kernel; it only extends that projection by zero on $\cH_M^\perp$.  Likewise, because $e_0 \in\cH_M$, the scalar $e_0$-spectral measure of $\widehat T_M$ is exactly that of the finite block $T_M$.

The vector $e_n$ is independent of $M$ once $M\ge n$, and every finitely supported vector belongs to $\cH_M$ for all sufficiently large $M$.  Consequently coefficientwise convergence may be tested on the common dense subspace $c_{00} = c_{00}(\Z_+)\subset\cH$ of finitely supported sequences.

\section{The \texorpdfstring{$0<q<1$}{0<q<1} right-edge scaling}
\label{sec:q-less-one-scaling}

We now specialize to the main fixed-$q$ regime
\[
\begin{gathered}
0<q<1, \qquad M \to \infty, \qquad N = N_M \to \infty, \qquad M-N_M \to \infty,\\
\rho_M := p_M q^{2N_M} \to \rho\in(0,\infty).
\end{gathered}
\]
As above, we write $N$ for $N_M$ when no confusion can arise.

\subsection{Centering and normalization}

Define the finite operator on $\cH_M$ by
\begin{equation}\label{eq:AM-definition}
A_M := p_M^{-1/2}(\Lambda_N I-H_{p_M,M}),
\end{equation}
and let $\widehat A_M = A_M \oplus 0$ on $\cH$ as in Section~\ref{subsec:common-hilbert-space}. The sign in this definition is chosen so that the occupied degrees lie on the positive side of the centered spectrum.  Indeed, from \eqref{eq:qk-finite-eigenvalue}, $\Lambda_{k+1}-\Lambda_k = (1-q)(q^{-k-1}+p_M q^k)>0$, so the eigenvalue of $A_M$ corresponding to $\psi_{k,M}$ is
\[
\eta_{k,M} = p_M^{-1/2}(\Lambda_N-\Lambda_k),
\]
and hence $\eta_{k,M}>0$ for $k<N$, $\eta_{N,M}=0$, and $\eta_{k,M}<0$ for $k>N$. Consequently, in the reflected coordinates of Section~\ref{sec:finite-spectral},
\begin{equation}\label{eq:finite-positive-projection}
\widehat K_{M,N} := K_{M,N} \oplus 0 = \ind_{(0,\infty)}(\widehat A_M).
\end{equation}

Using \eqref{eq:finite-reflected-offdiag}--\eqref{eq:finite-reflected-diag}, the Jacobi coefficients of $A_M$ are
\[
a_n^{(M)} = \sqrt{(q^{-n-1}-1)(1-q^{M-n})}, \qquad 0\le n<M,
\]
and
\[
b_n^{(M)} = p_M^{-1/2}(q^{-N}-q^{-n}) + p_M^{1/2}(q^{M-n}-q^N), \qquad 0\le n\le M.
\]
Thus
\[
(A_M f)_n = a_n^{(M)}f_{n+1} + a_{n-1}^{(M)}f_{n-1} + b_n^{(M)}f_n,
\]
with the $f_{-1}$ term omitted at $n=0$ and the $f_{M+1}$ term omitted at $n=M$.

\subsection{Coefficientwise limit}

\begin{proposition}\label{prop:q-less-one-coefficient-limit}
For every fixed $n\in\Z_+$,
\[
a_n^{(M)} \longrightarrow a_n := \sqrt{q^{-n-1}-1}, \qquad
b_n^{(M)} \longrightarrow c_\rho := \rho^{-1/2}-\rho^{1/2}.
\]
Equivalently, on the common dense domain $c_{00}\subset\ellp$ the finite Jacobi matrices converge coefficientwise to the formal Jacobi expression
\begin{equation}\label{eq:formal-qhermite-expression}
A_{\rho,\mathrm{formal}} = J+c_\rho I, \qquad (Jf)_n = a_n f_{n+1} + a_{n-1}f_{n-1}.
\end{equation}
\end{proposition}

\begin{proof}
For fixed $n$, the first limit follows immediately from $M-n\to\infty$: $1-q^{M-n} \longrightarrow 1$. For the diagonal coefficient, the scaling assumption gives
\[
p_M^{-1/2}q^{-N} = \rho_M^{-1/2} \longrightarrow \rho^{-1/2}, \qquad
p_M^{1/2}q^N = \rho_M^{1/2} \longrightarrow \rho^{1/2}.
\]
Moreover, for fixed $n$, $p_M^{-1/2}q^{-n} = \rho_M^{-1/2}q^{N-n} \longrightarrow 0$, while $p_M^{1/2}q^{M-n} = \rho_M^{1/2}q^{M-N-n} \longrightarrow 0$, because both $N\to\infty$ and $M-N\to\infty$.  Substituting these limits into the exact coefficients proves the claim.
\end{proof}

Here a \emph{Jacobi difference expression}, or simply \emph{Jacobi expression}, specifies the tridiagonal action on sequences, as in \eqref{eq:formal-qhermite-expression}, without specifying an operator domain in $\ell^2(\mathbb Z_+)$; see Teschl~\cite[\S\S1.1 and~2.6]{Teschl}. A \emph{self-adjoint realization} is a self-adjoint operator on this space whose action on its domain is given by that expression; see Bornemann~\cite{Bornemann} for this terminology. In other words, it is a self-adjoint extension of the associated minimal symmetric operator.

\subsection{Continuous \texorpdfstring{$q^{-1}$}{q-inverse}-Hermite identification}

The continuous $q^{-1}$-Hermite polynomials satisfy the recurrence
\[
2x h_n(x\mid q) = h_{n+1}(x\mid q) + q^{-n}(1-q^n)h_{n-1}(x\mid q);
\]
see, for example, \cite{CK,IM}.  After orthonormalization, multiplication by $2x$ has Jacobi off-diagonal coefficients $\sqrt{q^{-n-1}-1} = a_n$. Thus the unshifted expression $J$ in \eqref{eq:formal-qhermite-expression} is precisely the orthonormal continuous $q^{-1}$-Hermite Jacobi expression.  Writing its spectral variable as $s=2x$, the corresponding orthonormal first-kind polynomials are
\begin{equation}\label{eq:Pn-cont-qinv-Hermite}
P_n(s) = \frac{q^{n(n+1)/4}}{\sqrt{(q;q)_n}}h_n(s/2\mid q).
\end{equation}

Let $p_n^{(M)}$ denote the orthonormal first-kind polynomial determined by the Jacobi coefficients of $A_M$, normalized by $p_0^{(M)}=1$, and let $p_n^{(\rho)}$ denote the polynomial determined by the limiting coefficients in Proposition~\ref{prop:q-less-one-coefficient-limit}.  For every fixed $n$, the three-term recurrence involves only finitely many Jacobi coefficients, so Proposition~\ref{prop:q-less-one-coefficient-limit} gives $p_n^{(M)} \longrightarrow p_n^{(\rho)}$ coefficientwise as polynomials.  Since $A_{\rho,\mathrm{formal}} = J+c_\rho I$, the two limiting polynomials are related by
\begin{equation}\label{eq:shifted-first-kind-polynomials}
p_n^{(\rho)}(s+c_\rho) = P_n(s).
\end{equation}

These limits do not yet select a self-adjoint realization of the limiting Jacobi expression: the continuous $q^{-1}$-Hermite moment problem is indeterminate; see Section~\ref{sec:selected-extension} below.

\section{Direct convergence of the correlation kernels}
\label{sec:kernel-convergence}

We now use the finite dual representation of Proposition~\ref{prop:finite-dual-projection} to obtain the limiting correlation kernel directly.

For $l\in\Z$, set
\begin{equation}\label{eq:mu-l-explicit}
s_l := \rho^{1/2}q^{-l} - \rho^{-1/2}q^l, \qquad
\mu_l := \frac{\rho^{-2l}(1+\rho^{-1}q^{2l})q^{l(2l-1)}}
{(-\rho^{-1},-q\rho,q;q)_\infty}.
\end{equation}

\subsection{Convergence of the spectral weights}

For $l\in\mathcal I_M$, define
\[
\pi_M(l) := v_{0,M}(l)^2 = \psi_{N-l,M}(0)^2.
\]
Thus $\pi_M(l)$ is the $e_0$-spectral weight of $A_M$ at the eigenvalue corresponding to degree $N-l$.  Using the endpoint identity from Section~\ref{subsec:finite-dual-transform} and the standard squared norm in \cite[Eq.~(14.15.2)]{KLS} gives
\begin{equation}\label{eq:finite-dual-weight-formula}
\pi_M(l) = \frac{(q^{M-N+l+1};q)_{N-l}}{(q;q)_{N-l}} \frac{1+\rho_M q^{-2l}}
{(-q^{2l+1}/\rho_M;q)_{N-l} (-\rho_M q^{-2l};q)_{M-N+l+1}}.
\end{equation}

\begin{proposition}[Fixed-coordinate spectral weights]
\label{prop:finite-dual-weight-limit}
For every fixed $l\in\Z$,
\begin{equation}\label{eq:finite-dual-weight-limit}
\pi_M(l) \longrightarrow \mu_l.
\end{equation}
\end{proposition}

\begin{proof}
For fixed $l$, $N-l \longrightarrow \infty$, $M-N+l \longrightarrow \infty$, and $\rho_M \longrightarrow \rho$. Hence the finite products in \eqref{eq:finite-dual-weight-formula} converge to the corresponding infinite products, while $(q^{M-N+l+1};q)_{N-l} \longrightarrow 1$. Elementary $q$-shift identities for the two remaining infinite products then give exactly the expression for $\mu_l$ in \eqref{eq:mu-l-explicit}.
\end{proof}

\subsection{Convergence of the dual eigenvectors}

The reindexing \eqref{eq:dual-index} gives the exact centered eigenvalue
\[
\eta_{N-l,M} = \rho_M^{1/2}(q^{-l}-1) + \rho_M^{-1/2}(1-q^l),
\]
and therefore, for every fixed $l$,
\begin{equation}\label{eq:finite-dual-centered-eigenvalue-limit}
\eta_{N-l,M} \longrightarrow \lambda_l := s_l+c_\rho.
\end{equation}
Since $v_{0,M}(l)>0$, the normalized eigenvector at $\eta_{N-l,M}$ satisfies the exact identity
\begin{equation}\label{eq:finite-dual-polynomial-representation}
v_{n,M}(l) = \sqrt{\pi_M(l)}\, p_n^{(M)}(\eta_{N-l,M}).
\end{equation}
Combining \eqref{eq:finite-dual-weight-limit}, \eqref{eq:finite-dual-centered-eigenvalue-limit}, and the fixed-degree polynomial convergence from Section~\ref{sec:q-less-one-scaling}, including \eqref{eq:shifted-first-kind-polynomials}, gives
\[
v_{n,M}(l) \longrightarrow u_n(l) := \sqrt{\mu_l}\,P_n(s_l) \qquad (n,l\ \text{fixed}).
\]

The continuous $q^{-1}$-Hermite orthogonality and completeness of Christiansen--Koelink \cite[Theorem~3.9 and Eq.~(4.1)]{CK} show that $\{u_n\}_{n\ge0}$ is an orthonormal basis of $\ellZ$.  Their orthogonality formula is stated for a representative parameter in $(q,1]$; multiplying that parameter by a power of $q$ merely translates the lattice index $l$, so the same orthogonality and completeness hold for the present value $\rho^{-1/2}$.

We regard each $v_{n,M}$ as an element of $\ellZ$ by zero extension outside $\mathcal I_M$.  Finite dual orthogonality gives $\|v_{n,M}\|_{\ellZ} = 1$. The coordinatewise convergence of these normalized vectors to $u_n$ implies
\begin{equation}\label{eq:finite-dual-vector-strong-limit}
v_{n,M} \longrightarrow u_n \qquad \text{strongly in }\ellZ
\end{equation}
for every fixed $n$.

Define the limiting dual transform by
\begin{equation}\label{eq:bc-transform}
\cU_\rho:\ellp \longrightarrow \ellZ, \qquad \cU_\rho e_n = u_n.
\end{equation}
By the completeness just recalled, $\cU_\rho$ is unitary.

\subsection{Strong convergence of the projections}

Extend the finite dual transforms to operators $V_M:\ellp \longrightarrow \ellZ$ by
\[
V_M e_n =
\begin{cases}
v_{n,M},&0\le n\le M,\\
0,&n>M.
\end{cases}
\]
Then $\|V_M\|\le1$, and \eqref{eq:finite-dual-vector-strong-limit} gives convergence on every basis vector.  Hence
\[
V_M \longrightarrow \cU_\rho \qquad \text{strongly}.
\]

The adjoints converge strongly as well.  Indeed, $V_M^*V_M=\Pi_M$, where $\Pi_M$ is the orthogonal projection of $\ellp$ onto $\Span \{e_0,\ldots,e_M\}$, so $\Pi_M \to I$ strongly.  For $y=\cU_\rho x$,
\[
\|V_M^*y - \cU_\rho^*y\| = \|V_M^*\cU_\rho x - x\| \le \|\cU_\rho x - V_M x\| + \|\Pi_M x - x\| \longrightarrow 0.
\]
Since $\cU_\rho$ is onto, $V_M^* \longrightarrow \cU_\rho^*$ strongly.

By Proposition~\ref{prop:finite-dual-projection}, after zero extension to the common Hilbert space,
\[
\widehat K_{M,N} = V_M^*\ind_{\{l>0\}}V_M.
\]
Define
\begin{equation}\label{eq:bc-half-lattice-projection}
K_\rho := \cU_\rho^*\ind_{\{l>0\}}\cU_\rho.
\end{equation}

\begin{theorem}[Strong convergence of the projections]
\label{thm:projection-convergence}
Under the scaling assumptions of Section~\ref{sec:q-less-one-scaling},
\begin{equation}\label{eq:projection-convergence}
\widehat K_{M,N} \longrightarrow K_\rho \qquad \text{strongly on }\ellp.
\end{equation}
\end{theorem}

\begin{proof}
Since $\ind_{\{l>0\}}$ is a bounded multiplication operator, the strong convergences $V_M \to\cU_\rho$ and $V_M^* \to \cU_\rho^*$ imply
\[
V_M^*\ind_{\{l>0\}}V_M \longrightarrow \cU_\rho^*\ind_{\{l>0\}}\cU_\rho
\]
strongly.  This is \eqref{eq:projection-convergence}.
\end{proof}

Taking matrix entries gives the explicit limiting kernel
\begin{equation}\label{eq:spectral-series-kernel}
K_\rho(m,n) = \sum_{l=1}^{\infty} \mu_l P_m(s_l)P_n(s_l), \qquad m,n\in\Z_+.
\end{equation}
Consequently, for every fixed $m,n$, $K_{M,N}(m,n) \longrightarrow K_\rho(m,n)$.

\section{The selected self-adjoint realization}
\label{sec:selected-extension}

\subsection{Spectral realization of the direct limit}

We now identify the operator-theoretic meaning of the limiting transform $\cU_\rho$ and projection $K_\rho$ constructed in Section~\ref{sec:kernel-convergence}.
Let $\mathsf M_\rho$ be the multiplication operator on $\ellZ$ given by
\[
(\mathsf M_\rho f)(l) = \lambda_l f(l), \qquad
D(\mathsf M_\rho) = \left\{f\in\ellZ: \sum_{l\in\Z}\lambda_l^2|f(l)|^2<\infty\right\}.
\]
Since $\mathsf M_\rho$ is self-adjoint and $\cU_\rho$ is unitary,
\begin{equation}\label{eq:selected-extension-spectral-definition}
A_\rho := \cU_\rho^*\mathsf M_\rho\cU_\rho
\end{equation}
is self-adjoint on $\ellp$.

This operator realizes the formal Jacobi expression obtained in Section~\ref{sec:q-less-one-scaling}.  Indeed, the continuous $q^{-1}$-Hermite recurrence gives
\[
sP_n(s) = a_n P_{n+1}(s) + a_{n-1}P_{n-1}(s), \qquad a_{-1} := 0,
\]
where $a_n = \sqrt{q^{-n-1}-1}$.  Hence, using $u_n(l) = \sqrt{\mu_l}\,P_n(s_l)$ and $\lambda_l = s_l+c_\rho$,
\[
\mathsf M_\rho u_n = a_n u_{n+1} + c_\rho u_n + a_{n-1}u_{n-1}.
\]
Equivalently,
\[
A_\rho e_n = a_n e_{n+1} + c_\rho e_n + a_{n-1}e_{n-1}, \qquad n\ge0.
\]
Thus $A_\rho$ agrees on finitely supported vectors with $A_{\rho,\mathrm{formal}}$, and is therefore a self-adjoint realization of the limiting Jacobi expression.

Since $P_0 \equiv1$, $(\cU_\rho e_0)(l) = u_0(l) = \sqrt{\mu_l}$, and consequently the spectral measure of $A_\rho$ at $e_0$ is
\begin{equation}\label{eq:selected-spectral-measure}
\sum_{l\in\Z}\mu_l\,\delta_{\lambda_l}.
\end{equation}
In particular, $\lambda_0=0$.  Since $s_l$ is strictly increasing in $l$, $\lambda_l>0 \Longleftrightarrow l>0$. Therefore
\begin{equation}\label{eq:limiting-correlation-operator}
\ind_{(0,\infty)}(A_\rho) = \cU_\rho^*\ind_{\{l>0\}}\cU_\rho = K_\rho.
\end{equation}

\subsection{The limit-circle characterization}

We next identify the self-adjoint realization constructed above within the one-parameter family associated with the limiting Jacobi expression. Let $J_0$ be the operator defined by $J$ in \eqref{eq:formal-qhermite-expression} on $c_{00}(\Z_+)$, and put
\[
J_{\min} := \overline{J_0}, \qquad J_{\max} := J_{\min}^*.
\]
For $0<q<1$, the Hamburger moment problem for the continuous $q^{-1}$-Hermite polynomials is indeterminate; see Ismail--Masson \cite{IM}. By the equivalence between determinacy of the Hamburger moment problem and essential self-adjointness of the associated minimal Jacobi operator, $J_{\min}$ has deficiency indices
\[
n_+(J_{\min}) = n_-(J_{\min}) = 1;
\]
see Simon~\cite[Theorem~2; Theorems~2.10 and~2.12]{Simon}. Equivalently, the Jacobi expression is limit-circle at $+\infty$. The shifted minimal operator
\[
A_{\rho,\min} := J_{\min}+c_\rho I
\]
has the same deficiency indices and hence a one-parameter family of self-adjoint extensions.

For $z\in\C$, the first-kind polynomials introduced in Section~\ref{sec:q-less-one-scaling} form the solution
\[
p^{(\rho)}(z) := \{p_n^{(\rho)}(z)\}_{n\ge0}
\]
of
\[
A_{\rho,\mathrm{formal}}p^{(\rho)}(z) = zp^{(\rho)}(z), \qquad p_0^{(\rho)}(z) = 1.
\]
In the limit-circle case $p^{(\rho)}(z)\in\ellp$ for every $z\in\C$; see Simon~\cite[Theorem~3]{Simon}.

The spectral realization of the preceding subsection shows directly which extension occurs here. Since $\lambda_0=0$, the vector $\cU_\rho^*\delta_0$ is a zero eigenvector of $A_\rho$. Moreover, using $s_0=-c_\rho$ and \eqref{eq:shifted-first-kind-polynomials},
\[
\langle e_n,\cU_\rho^*\delta_0\rangle = u_n(0) = \sqrt{\mu_0}\,P_n(s_0) =
\sqrt{\mu_0}\,p_n^{(\rho)}(0).
\]
Hence
\begin{equation}\label{eq:selected-zero-mode}
\cU_\rho^*\delta_0 = \sqrt{\mu_0}\,p^{(\rho)}(0),
\end{equation}
so $A_\rho$ is a self-adjoint extension of $A_{\rho,\min}$ containing the square-summable formal zero mode.

For completeness, we record its boundary characterization. Put
\[
D(A_{\rho,\max}) = \{f\in\ellp: A_{\rho,\mathrm{formal}}f\in\ellp\}, \qquad
A_{\rho,\max} := A_{\rho,\min}^*,
\]
and define the discrete Wronskian
\[
W(f,g)(n) := a_n \bigl(f_{n+1}g_n - f_n g_{n+1}\bigr).
\]
For $f,g\in D(A_{\rho,\max})$, Green's identity gives the boundary form
\[
W_\infty(\overline f,g) := \lim_{n\to\infty}W(\overline f,g)(n);
\]
see Simon~\cite[Theorem~2.7, especially (2.8) and (2.14)]{Simon}. The extension $A_\rho$ is therefore characterized by
\begin{equation}\label{eq:zero-boundary-extension}
D(A_\rho) = \{f\in D(A_{\rho,\max}): W_\infty(\overline f,p^{(\rho)}(0)) = 0\}.
\end{equation}
Equivalently,
\[
D(A_\rho) = D(A_{\rho,\min}) + \C p^{(\rho)}(0);
\]
compare Simon~\cite[(4.20b) and (4.24)]{Simon}.

Distinct self-adjoint extensions have disjoint spectra, and every real point belongs to the spectrum of exactly one such extension; see Simon~\cite[Theorem~4.11]{Simon}. Thus $A_\rho$ is the unique self-adjoint extension of $A_{\rho,\min}$ having $0$ as an eigenvalue.

\subsection{The maximal zero atom}

The orthogonality and completeness used in Section~\ref{sec:kernel-convergence} show that $\sum_{l\in\Z}\mu_l\,\delta_{s_l}$ is an N-extremal measure for the unshifted continuous $q^{-1}$-Hermite moment problem.  Translation by $c_\rho$ preserves this property, so by \eqref{eq:selected-spectral-measure} the $e_0$-spectral measure
\[
\nu_\rho := \sum_{l\in\Z}\mu_l\,\delta_{\lambda_l}
\]
of $A_\rho$ is N-extremal for the shifted moment problem.

Let $\mathfrak M_\rho$ denote the set of all representing measures for this shifted moment problem, and define
\[
M_\rho(x) := \sup_{\nu\in\mathfrak M_\rho}\nu(\{x\}), \qquad x\in\R.
\]
For an indeterminate Hamburger moment problem, Simon's Theorem~5 states that for every real $x$ there is a unique N-extremal measure carrying an atom at $x$, and that its mass there is strictly larger than the mass assigned to $x$ by any other representing measure.  Moreover, if $p_n^{(\rho)}(x)$ are the first-kind orthonormal polynomials of the shifted problem, then Simon's Theorem~4.11 and Eq.~(4.25) give
\begin{equation}\label{eq:maximal-atom-formula-rho}
M_\rho(x) = \left(\sum_{n=0}^{\infty}|p_n^{(\rho)}(x)|^2\right)^{-1}.
\end{equation}

\begin{proposition}[Maximal zero atom]\label{prop:candidate-extension}
The spectral measure of $A_\rho$ saturates the universal maximal-atom bound at zero:
\begin{equation}\label{eq:selected-zero-mass}
\mu_0 = M_\rho(0) = \left(\sum_{n=0}^{\infty}|p_n^{(\rho)}(0)|^2\right)^{-1}.
\end{equation}
\end{proposition}

\begin{proof}
By \eqref{eq:finite-dual-centered-eigenvalue-limit}, $\lambda_0=0$, so the N-extremal spectral measure $\nu_\rho$ contains zero with mass $\mu_0$.  The uniqueness and maximality statement in Simon's Theorem~5, together with \eqref{eq:maximal-atom-formula-rho}, therefore gives \eqref{eq:selected-zero-mass}.
\end{proof}

\section{A maximal-atom selection principle}
\label{sec:selection-principle}

The direct argument of Section~\ref{sec:kernel-convergence} used the full fixed-coordinate profile $\pi_M(l) \longrightarrow \mu_l$ for fixed $l\in\Z$ to obtain the limiting correlation kernel. For selection of the corresponding self-adjoint realization, only a single spectral mass is needed. We isolate this mechanism abstractly and show that this single saturating atom already forces strong-resolvent convergence to the selected realization. The criterion is sharp in the following sense: saturation of the maximal-atom bound at the limiting point is equivalent to the limiting measure being the unique N-extremal measure whose support contains that point. Thus the single-atom principle captures precisely the situation in which the surviving atom selects an N-extremal spectral measure, as it does at the $q$-Krawtchouk edge.

Let $T_M$ be finite Jacobi matrices on $\cH_M$, let $\widehat T_M = T_M \oplus 0$ on $\cH$ as in Section~\ref{subsec:common-hilbert-space}, and let $\nu_M$ denote the spectral measure of $T_M$ corresponding to the cyclic vector $e_0$, equivalently the $e_0$-spectral measure of $\widehat T_M$.

Let $\tau$ be a half-line Jacobi expression with coefficients $a_n,b_n$, and suppose that its minimal operator $T_{\min}$ is in the limit-circle case, with deficiency indices $(1,1)$. Let $p_n$ be the associated orthonormal polynomials and let $\mathfrak M$ be the set of all representing measures for the corresponding indeterminate Hamburger moment problem. Define
\[
M(x) := \sup_{\nu\in\mathfrak M}\nu(\{x\}) = \left(\sum_{n=0}^{\infty}|p_n(x)|^2\right)^{-1}.
\]
For each $x\in\R$, let $\nu^{[x]}$ denote the unique N-extremal representing measure whose support contains $x$, and let $T^{[x]}$ be the corresponding self-adjoint extension of $T_{\min}$. Then
\[
\nu^{[x]}(\{x\}) = M(x).
\]
These maximal-atom and equality statements follow from Simon's Theorem~5, Theorem~4.11 and Eq.~(4.25), together with Theorem~4.17; the correspondence with self-adjoint extensions follows from Simon's construction of the von Neumann solutions; see \cite[Theorem~4.10, Proposition~4.15, and Theorem~5]{Simon}.

\begin{theorem}[Self-adjoint extension selection by a moving maximal atom]
\label{thm:selection}
Let $T_M$ be finite nondegenerate Jacobi matrices and assume that, for every fixed $n$,
\[
a_n^{(M)} \to a_n>0, \qquad b_n^{(M)} \to b_n.
\]
Suppose that there are distinguished spectral points $x_M \in\R$ and $x_\infty\in\R$ such that
\begin{equation}\label{eq:moving-maximal-atom-hypothesis}
\boxed{ x_M \longrightarrow x_\infty, \qquad \nu_M(\{x_M\}) \longrightarrow M(x_\infty). }
\end{equation}
Then
\[
\widehat T_M \xrightarrow{\mathrm{s.r.}} T^{[x_\infty]}.
\]
\end{theorem}

\begin{proof}
For every fixed integer $k\ge0$, the spectral theorem gives
\[
\int_{\R}x^k\,d\nu_M(x) = \langle e_0,T_M^k e_0\rangle .
\]
The vector $T_M^k e_0$ involves only finitely many Jacobi coefficients. Convergence of those coefficients therefore gives convergence of every fixed moment, and in particular uniform boundedness of every fixed even moment. The second-moment bound gives tightness, while a higher even-moment bound gives the uniform integrability needed to pass any fixed moment through weak convergence.

Take a weakly convergent subsequence $\nu_{M_j} \Rightarrow \nu$. By the preceding moment convergence and uniform integrability, $\nu$ has exactly the moments determined by the limiting Jacobi coefficients. Hence $\nu$ is a representing measure for the limiting Hamburger moment problem. At this stage no N-extremality is asserted: an arbitrary representing measure need not be the spectral measure of a self-adjoint extension of $T_{\min}$.

Fix $\varepsilon>0$ and put
\[
F_\varepsilon = [x_\infty-\varepsilon,x_\infty+\varepsilon].
\]
For all sufficiently large $j$, $x_{M_j}\in F_\varepsilon$, so $\nu_{M_j}(F_\varepsilon) \ge \nu_{M_j}(\{x_{M_j}\})$. Since $F_\varepsilon$ is closed, the Portmanteau theorem and \eqref{eq:moving-maximal-atom-hypothesis} give
\[
\nu(F_\varepsilon) \ge \limsup_{j\to\infty}\nu_{M_j}(F_\varepsilon)
\ge \lim_{j\to\infty}\nu_{M_j}(\{x_{M_j}\}) = M(x_\infty).
\]
Letting $\varepsilon \downarrow 0$ and using continuity from above yields $\nu(\{x_\infty\})\ge M(x_\infty)$. By definition of the maximal atom mass, every representing measure satisfies $\nu(\{x_\infty\})\le M(x_\infty)$. Consequently
\[
\nu(\{x_\infty\}) = M(x_\infty).
\]
Simon's equality case for the maximal-atom bound gives $\nu = \nu^{[x_\infty]}$; see Theorem~5 together with Theorems~4.11 and~4.17. Thus every weakly convergent subsequence has the same limit, and therefore
\[
\nu_M \Rightarrow \nu^{[x_\infty]}.
\]

It remains to pass from scalar spectral measures to resolvents. The vector $e_0$ is cyclic for the finite Jacobi block $T_M$ on $\cH_M$. If $p_j^{(M)}$ denotes the orthonormal polynomial determined by the finite Jacobi coefficients, then for every fixed $j$,
\[
e_j = p_j^{(M)}(T_M)e_0 = p_j^{(M)}(\widehat T_M)e_0.
\]
Thus, for fixed $j,k$ and $z\in\C \setminus \R$,
\[
\langle e_j,(\widehat T_M-z)^{-1}e_k\rangle = \int_{\R}
\frac{p_j^{(M)}(x)p_k^{(M)}(x)}{x-z}\,d\nu_M(x).
\]
Because the coefficients of the fixed-degree polynomials $p_j^{(M)}$ and $p_k^{(M)}$ converge, they are uniformly bounded coefficientwise for all sufficiently large $M$. Hence there is a constant $C_{j,k,z}$, independent of such $M$, for which
\[
\left| \frac{p_j^{(M)}(x)p_k^{(M)}(x)}{x-z} \right|
\le C_{j,k,z}(1 + |x|^{j+k}), \qquad x\in\R.
\]
Choose an integer $r$ such that $2r>j+k$. The already established moment bounds give, for some $C_r<\infty$, $\sup_M \int_{\R}x^{2r}\,d\nu_M(x)\le C_r$, and therefore
\[
\sup_M \int_{|x|>R}(1 + |x|^{j+k})\,d\nu_M(x) \le C_r \bigl(R^{-2r} + R^{j+k-2r}\bigr)
\xrightarrow[R \to \infty]{} 0.
\]
Thus the resolvent integrands are uniformly integrable. On every compact real interval they converge uniformly, since $p_j^{(M)} \to p_j$ and $p_k^{(M)} \to p_k$ coefficientwise and $z\notin\R$. Weak convergence of $\nu_M$ therefore gives, by linearity, for $f,g\in c_{00}$,
\[
\langle f,(\widehat T_M-z)^{-1}g\rangle \longrightarrow
\langle f,(T^{[x_\infty]}-z)^{-1}g\rangle.
\]
The uniform resolvent bound $\|(\widehat T_M-z)^{-1}\|\le|\Im z|^{-1}$ and density of $c_{00}$ extend this to weak resolvent convergence on all of $\cH$.

Finally take $z\in\C_+$. For every $f\in\cH$,
\[
\|(\widehat T_M-z)^{-1}f\|^2 = \frac{1}{\Im z}
\Im \langle f,(\widehat T_M-z)^{-1}f\rangle,
\]
and the same identity holds for $T^{[x_\infty]}$. Weak resolvent convergence therefore implies convergence of the corresponding resolvent-vector norms. Weak convergence together with convergence of norms is strong convergence, so
\[
(\widehat T_M-z)^{-1}f \longrightarrow (T^{[x_\infty]}-z)^{-1}f
\]
for every $f\in\cH$, which is the claimed strong resolvent convergence.
\end{proof}

\begin{remark}[The moving maximal-atom mechanism]
The decisive input is saturation of the maximal-atom bound, not merely the survival of a positive atom. A positive limiting atom may occur in a non-N-extremal representing measure and therefore need not select a self-adjoint extension.
\end{remark}

\begin{corollary}[Strong resolvent convergence for the $q$-Krawtchouk limit]
\label{cor:qk-strong-resolvent}
Assume $0<q<1$ and
\[
M \to \infty, \qquad N = N_M \to \infty, \qquad M-N_M \to \infty, \qquad
p_M q^{2N_M} \to \rho\in(0,\infty).
\]
Let $A_M$ be the reflected and centered $q$-Krawtchouk Jacobi operators defined in Section~\ref{sec:q-less-one-scaling}, let $\widehat A_M = A_M \oplus 0$ on $\cH$, and let $A_\rho$ be the self-adjoint realization identified in Section~\ref{sec:selected-extension}. Then
\[
\widehat A_M \xrightarrow{\mathrm{s.r.}} A_\rho.
\]
\end{corollary}

\begin{proof}
Proposition~\ref{prop:q-less-one-coefficient-limit} gives the coefficient convergence required by Theorem~\ref{thm:selection}. Take $x_M=x_\infty=0$. Since the degree-$N$ eigenvalue of $A_M$ is zero, $\nu_M(\{0\}) = |\psi_{N,M}(0)|^2 = \pi_M(0)$. By Proposition~\ref{prop:finite-dual-weight-limit} and Proposition~\ref{prop:candidate-extension},
\[
\nu_M(\{0\}) = \pi_M(0) \longrightarrow \mu_0 = M_\rho(0).
\]
The unique N-extremal extension containing zero is precisely $A_\rho$ by Section~\ref{sec:selected-extension}. Hence $\widehat A_M \xrightarrow{\mathrm{s.r.}} A_\rho$.
\end{proof}

\section{Dual spectral structure and explicit formulas}\label{sec:explicit-kernel}

The preceding sections give two complementary descriptions of the limiting projection,
\[
K_\rho = \ind_{(0,\infty)}(A_\rho) = \cU_\rho^*\,\ind_{\{l>0\}}\,\cU_\rho.
\]
We develop the dual operator behind the second representation and use it to obtain explicit formulas for the kernel.

\subsection{The dual operator and the Borodin--Corwin picture}\label{subsec:dual-bc}

This is the fixed-$q$, discrete-spectral analogue of the bispectral picture underlying the Borodin--Olshanski construction \cite{BO}.  If $\mathcal F_{\mathrm{cl}}$ denotes the orthogonal-polynomial transform for a classical weight and $J_{\mathrm{cl}}$ is the Jacobi matrix representing multiplication by the continuous variable $t$, then, for the positive half-line ensemble, their construction may be written as
\[
\ind_{(0,\infty)}(J_{\mathrm{cl}}-r) =
\mathcal F_{\mathrm{cl}}^*\,\ind_{\{t>r\}}\,\mathcal F_{\mathrm{cl}}.
\]
Thus the same projection is viewed spectrally in the polynomial-index representation and spatially in the continuous spectral representation.  In the Hermite, Laguerre, and Jacobi cases, Borodin--Olshanski~\cite[\S\S3.2--3.4]{BO} use forward and backward shift identities together with integration by parts to reduce the half-line integrals to boundary terms and obtain the integrable form of the kernel.  In the present fixed-$q$ limit the spectral variable remains discrete, so the half-line is replaced by the half-lattice $l>0$.  Correspondingly, the continuous-variable relations are replaced here by a second-order difference operator in $l$, whose discrete Green identity produces the analogous boundary reduction.

We now return to the finite dual $q$-Krawtchouk transform of Section~\ref{subsec:finite-dual-transform}.  Let $\mathcal D_M$ denote the dual $q$-Krawtchouk difference operator in the variable $k$, with parameter $c_M=-p_M q^M$.  Explicitly,
\[
(\mathcal D_M y)(k) = B_M(k)y(k+1) - [B_M(k) + D_M(k)]y(k) + D_M(k)y(k-1),
\]
where
\begin{align*}
B_M(k) & = \frac{(1-q^{k-M})(1+p_M q^k)}
{(1+p_M q^{2k})(1+p_M q^{2k+1})},\\
D_M(k) & = -\frac{p_M q^{2k-M-1}(1-q^k)(1+p_M q^{M+k})}
{(1+p_M q^{2k-1})(1+p_M q^{2k})}.
\end{align*}
Using \eqref{eq:qk-duality-kls}, the eigenvalue relation \cite[Eq.~(14.17.5)]{KLS} takes the form
\[
\mathcal D_M K_k(q^{-x};p_M,M;q) = q^{-x}(1-q^x)K_k(q^{-x};p_M,M;q), \qquad x = 0,\ldots,M.
\]
For the reflected spatial site $n$, we have $x=M-n$ in \eqref{eq:qk-duality-kls}. Since $q^M\left[1 + q^{-(M-n)}(1-q^{M-n})\right] = q^n$, the operator $q^M(I+\mathcal D_M)$ has the exact eigenvalues $q^n$, $n=0,\ldots,M$.

After the reindexing $k=N-l$, symmetrization by the dual orthogonality weight, and the alternating gauge in \eqref{eq:finite-dual-vectors}, this operator is
\begin{equation}\label{eq:finite-dual-operator-spectral}
\sL_M := \cU_M\,\diag (1,q,\ldots,q^M)\,\cU_M^*.
\end{equation}
For interior $l\in\mathcal I_M$, set
\[
r^-_{l,M} := -q^M B_M(N-l), \qquad r^+_{l,M} := -q^M D_M(N-l).
\]
Then $r^\pm_{l,M}>0$, and $\sL_M$ is the positive-off-diagonal Jacobi matrix with
\begin{equation}\label{eq:finite-dual-jacobi-coeffs}
a^{\dual,M}_l = \sqrt{r^+_{l,M}r^-_{l+1,M}}, \qquad
b^{\dual,M}_l = q^M + r^-_{l,M} + r^+_{l,M}.
\end{equation}
Define
\begin{align}
a_l^{\dual} & = \frac{\rho q^{2l+1/2}}
{(\rho+q^{2l+1})
\sqrt{(\rho+q^{2l})(\rho+q^{2l+2})}},\label{eq:bc-dual-a}\\
b_l^{\dual} & = \frac{\rho(1+q)q^{2l-1}}
{(\rho+q^{2l-1})(\rho+q^{2l+1})}.
\label{eq:bc-dual-b}
\end{align}
Let $\sL_\rho^{\dual}$ be the bilateral Jacobi operator on $\ellZ$
with these coefficients. They are bounded, so $\sL_\rho^{\dual}$
is a bounded symmetric, hence self-adjoint, operator.

\begin{proposition}[Finite dual operator limit]
\label{prop:finite-dual-operator-limit}
For each fixed $l\in\Z$, the coefficients of the finite
positive-gauge dual $q$-Krawtchouk operator $\sL_M$ converge
to the coefficients of $\sL_\rho^{\dual}$.
\end{proposition}

\begin{proof}
Using $\rho_M = p_M q^{2N}$, we obtain
\begin{align*}
r^-_{l,M} & = \frac{(q^{N-l}-q^M)(1+\rho_M q^{-N-l})}
{(1+\rho_M q^{-2l})(1+\rho_M q^{-2l+1})},\\
r^+_{l,M} & = \frac{\rho_M q^{-2l-1}(1-q^{N-l})
(1+\rho_M q^{M-N-l})}
{(1+\rho_M q^{-2l-1})(1+\rho_M q^{-2l})}.
\end{align*}
Consequently, for every fixed $l\in\Z$,
\begin{align*}
r^-_{l,M} & \longrightarrow
\frac{\rho q^{2l-1}}
{(\rho+q^{2l})(\rho+q^{2l-1})},\\
r^+_{l,M} & \longrightarrow
\frac{\rho q^{2l}}
{(\rho+q^{2l})(\rho+q^{2l+1})}.
\end{align*}
Substituting these limits into
\eqref{eq:finite-dual-jacobi-coeffs} and using $q^M\to0$
gives \eqref{eq:bc-dual-a}--\eqref{eq:bc-dual-b}.
\end{proof}

The strong convergence \eqref{eq:finite-dual-vector-strong-limit} also identifies the limiting eigenfunctions of this operator.  For fixed $l$, the finite eigenvalue equation $\sL_M v_{n,M} = q^n v_{n,M}$ involves only the coefficients and vector entries at $l-1,l,l+1$.  Passing to the limit using Proposition~\ref{prop:finite-dual-operator-limit} therefore gives
\begin{equation}\label{eq:dual-limit-eigenrelation}
\sL_\rho^{\dual}u_n = q^n u_n, \qquad n\ge0.
\end{equation}
Since $\{u_n\}_{n\ge0}$ is an orthonormal basis and \eqref{eq:bc-transform} defines $\cU_\rho$, it follows that
\[
\sL_\rho^{\dual} = \cU_\rho\,\diag (1,q,q^2,\ldots)\,\cU_\rho^*.
\]

The ground-state transform gives a bijection between finite irreducible symmetric matrices with nonnegative entries and Perron eigenvalue $1$, and finite irreducible reversible Markov matrices; see Levin--Peres~\cite[\S12.1, proof of Lemma~12.2]{LP}. For $\sL_M$, the positive normalized eigenvector at eigenvalue $1$ is $v_{0,M}$. Its ground-state transform
\[
P_M(l,l') := \frac{v_{0,M}(l')}{v_{0,M}(l)}\,\sL_M(l,l')
\]
is a reversible birth--death Markov chain on $\mathcal I_M$, with stationary distribution $\pi_M(l) = v_{0,M}(l)^2$.

A corresponding limiting Markov picture appears in Borodin--Corwin \cite{BC}.  After identifying their dynamic-ASEP parameter $\alpha_{\mathrm{BC}}$ with our $\rho$, their one-step spatial Markov operator is
\begin{equation}\label{eq:bc-one-step}
(\mathsf Kf)(s) = \frac{q^s}{\rho+q^s}f(s+1) + \frac{\rho}{\rho+q^s}f(s-1), \qquad s\in\Z;
\end{equation}
see \cite[Definition~2.5, Eq.~(2.4)]{BC}.  Their spectral coordinate is
\[
\frac12\mathbf f(s) = \frac12\left(\rho^{1/2}q^{-s/2} - \rho^{-1/2}q^{s/2}\right),
\]
see~\cite[Eq.~(2.8)]{BC}.  Thus on the even lattice $s=2l$ it is exactly $\frac12s_l$.  Moreover, their even-parity stationary marginal is
\begin{equation}\label{eq:bc-stationary-mass}
\mathbb P(s_0 = 2l) = \frac{\rho^{-2l}q^{l(2l-1)}(1+\rho^{-1}q^{2l})}
{(-\rho^{-1},-q\rho,q;q)_\infty} = \mu_l;
\end{equation}
see \cite[Definition~2.12, Lemma~2.14, and Remark~2.17]{BC}.  Thus the stationary marginal in \cite{BC} is the same N-extremal continuous $q^{-1}$-Hermite measure selected in Section~\ref{sec:selected-extension}.

The operator $\mathsf K$ changes parity.  Let $\sP$ be the two-step chain induced by $\mathsf K^2$ on the even lattice, written in the coordinate $l=s/2$.  Its transition probabilities are
\begin{align}
\sP(l,l+1) & = \frac{q^{4l+1}}{(\rho+q^{2l})(\rho+q^{2l+1})},\label{eq:bc-P-plus}\\
\sP(l,l-1) & = \frac{\rho^2}{(\rho+q^{2l})(\rho+q^{2l-1})},\label{eq:bc-P-minus}\\
\sP(l,l) & = \frac{\rho(1+q)q^{2l-1}}
{(\rho+q^{2l-1})(\rho+q^{2l+1})}.\label{eq:bc-P-diag}
\end{align}
Borodin--Corwin~\cite[Lemma~2.9]{BC} gives the one-step eigenvalue $q^{n/2}$; restricting the square of their operator to the even parity class, \cite[Eq.~(4.8)]{BC} gives
\begin{equation}\label{eq:bc-P-eigenrelation}
\sP\,h_n\!\left(\frac{s_l}{2}\mid q\right) = q^n h_n\!\left(\frac{s_l}{2}\mid q\right).
\end{equation}
The detailed-balance relation for the measure $\mu = \{\mu_l\}_{l\in\Z}$ can be checked directly.  Indeed, from \eqref{eq:bc-stationary-mass}, \eqref{eq:bc-P-plus}, and \eqref{eq:bc-P-minus},
\[
\frac{\mu_{l+1}}{\mu_l} = \rho^{-2}q^{4l+1}\frac{\rho+q^{2l+2}}{\rho+q^{2l}} =
\frac{\sP(l,l+1)}{\sP(l+1,l)}.
\]
Hence $\mu$ is reversible for $\sP$.  In particular, $\sP$ is symmetric on $\ell^2(\Z,\mu)$.  Reversibility also implies that $\mu$ is stationary, and, since $\sP$ is a Markov operator, Jensen's inequality gives $\|\sP\|\le1$.  Therefore $\sP$ is a bounded everywhere-defined symmetric operator, and hence is self-adjoint.

With
\[
U_\mu:\ell^2(\Z,\mu) \longrightarrow \ellZ, \qquad (U_\mu f)(l) = \sqrt{\mu_l}\,f(l),
\]
put
\begin{equation}\label{eq:bc-symmetrized-operator}
\sL_{\mathrm{BC}} := U_\mu\sP U_\mu^{-1}.
\end{equation}
This is self-adjoint on $\ellZ$.  Its diagonal coefficient is \eqref{eq:bc-P-diag}, and its off-diagonal coefficient is
\[
\sqrt{\sP(l,l+1)\sP(l+1,l)} = \frac{\rho q^{2l+1/2}}
{(\rho+q^{2l+1})\sqrt{(\rho+q^{2l})(\rho+q^{2l+2})}}.
\]
Thus its Jacobi coefficients are exactly \eqref{eq:bc-dual-a}--\eqref{eq:bc-dual-b}, and therefore
\begin{equation}\label{eq:bc-dual-identification}
\sL_{\mathrm{BC}} = \sL_\rho^{\dual}.
\end{equation}

The convergence of the Jacobi coefficients, together with
the weight convergence $\pi_M(l) \longrightarrow \mu_l>0$
from Proposition~\ref{prop:finite-dual-weight-limit}, shows
that the finite ground-state transforms converge entrywise
to the Borodin--Corwin two-step transition operator:
for fixed $l,l'\in\Z$,
\[
P_M(l,l') \longrightarrow
\sqrt{\frac{\mu_{l'}}{\mu_l}}\,\sL_\rho^{\dual}(l,l')
= \sP(l,l').
\]

\subsection{Christoffel--Darboux / integrable form}
The limiting dual operator from the preceding subsection gives an efficient summation of \eqref{eq:spectral-series-kernel} away from the diagonal.  For brevity write $\sL := \sL_\rho^{\dual} = \sL_{\mathrm{BC}}$, so that
\[
(\sL f)_l = a_l^{\dual}f_{l+1} + b_l^{\dual}f_l + a_{l-1}^{\dual}f_{l-1},
\]
with $a_l^{\dual},b_l^{\dual}$ given by \eqref{eq:bc-dual-a}--\eqref{eq:bc-dual-b}.  For the eigenvectors $u_n$ from Section~\ref{sec:kernel-convergence}, $\sL u_n=q^n u_n$. For sequences $f,g$ define the Casoratian
\[
[f,g]_l := a_l^{\dual}\bigl(f(l+1)g(l) - f(l)g(l+1)\bigr).
\]
The discrete Green identity is
\[
\bigl((\sL f)_l g(l) - f(l)(\sL g)_l\bigr) = [f,g]_l - [f,g]_{l-1}.
\]
Applying this to $u_m,u_n$ and summing over $l=1,\ldots,M$ gives
\[
(q^m-q^n)\sum_{l=1}^M u_m(l)u_n(l) = [u_m,u_n]_M - [u_m,u_n]_0.
\]
For fixed $n$, the explicit mass formula and the degree of $P_n$ imply
\[
u_n(l) = O\!\left(\rho^{-l}q^{l^2-(n+1/2)l}\right), \qquad a_l^{\dual} = O(q^{2l}), \qquad
l \to +\infty,
\]
so the upper boundary term tends to zero super-exponentially.  Letting $M\to\infty$ and using \eqref{eq:spectral-series-kernel}, we obtain
\[
K_\rho(m,n) = \frac{a_0^{\dual}}{q^m-q^n} \bigl(
u_m(0)u_n(1) - u_m(1)u_n(0)\bigr), \qquad m\ne n.
\]
From the explicit formulas for $a_0^{\dual}$, $\mu_0$, and $\mu_1$,
\[
a_0^{\dual}\sqrt{\mu_0 \mu_1} = \frac{q}
{\rho(\rho+q)(-\rho^{-1},-q\rho,q;q)_\infty} =: \varkappa_\rho.
\]
With $s_l=s_l(\rho)$ as in \eqref{eq:mu-l-explicit}, set
\[
F_\rho(n) := \sqrt{\varkappa_\rho}\, \frac{q^{n(n+1)/4}}{\sqrt{(q;q)_n}}\, h_n(s_0(\rho)/2\mid q), \qquad
G_\rho(n) := \sqrt{\varkappa_\rho}\, \frac{q^{n(n+1)/4}}{\sqrt{(q;q)_n}}\,
h_n(s_1(\rho)/2\mid q).
\]

\begin{proposition}[Christoffel--Darboux form of the limiting kernel]\label{prop:cd}
For $m,n\in\Z_+$ with $m\ne n$,
\begin{equation}\label{eq:cd-kernel}
K_\rho(m,n) = \frac{F_\rho(m)G_\rho(n) - G_\rho(m)F_\rho(n)}{q^m-q^n}.
\end{equation}
Equivalently,
\[
K_\rho(m,n) = \varkappa_\rho \frac{q^{[m(m+1)+n(n+1)]/4}}{\sqrt{(q;q)_m(q;q)_n}} \frac{
h_m(s_0(\rho)/2\mid q)h_n(s_1(\rho)/2\mid q) - h_m(s_1(\rho)/2\mid q)h_n(s_0(\rho)/2\mid q)}
{q^m-q^n}.
\]
For $m=n$ we retain the spectral series \eqref{eq:spectral-series-kernel}.
\end{proposition}

\subsection{A partial-theta expansion}
The half-lattice spectral sum can be written as a finite combination of partial theta functions.  For $0\le j\le n$, write
\[
\genfrac{[}{]}{0pt}{}{n}{j}_{q^{-1}} := \frac{(q^{-1};q^{-1})_n}
{(q^{-1};q^{-1})_j\,(q^{-1};q^{-1})_{n-j}}
\]
for the $q^{-1}$-binomial coefficient.  By the classical Laurent expansion for continuous $q$-Hermite polynomials, specialized to base $q^{-1}$ (see Koekoek--Lesky--Swarttouw~\cite[\S14.26, Remark after (14.26.14)]{KLS}), we have
\[
h_n\!\left(\frac{s_l}{2}\mid q\right) = \sum_{j=0}^n (-1)^j
\genfrac{[}{]}{0pt}{}{n}{j}_{q^{-1}} \rho^{-j+n/2}q^{l(2j-n)}.
\]
Define
\[
\Theta_q(z) := \sum_{l=1}^\infty q^{2l^2}z^l, \qquad
D_r^{m,n} := (-1)^r \sum_{\substack{j+k=r\\0\le j\le m,\ 0\le k\le n}}
\genfrac{[}{]}{0pt}{}{m}{j}_{q^{-1}} \genfrac{[}{]}{0pt}{}{n}{k}_{q^{-1}}.
\]
Equivalently, $D_r^{m,n} = [t^r]H_m(-t;q^{-1})H_n(-t;q^{-1})$, where
\[
H_n(t;q^{-1}) := \sum_{j=0}^n \genfrac{[}{]}{0pt}{}{n}{j}_{q^{-1}}t^j
\]
is the Rogers--Szeg\H{o} polynomial.

\begin{proposition}[Partial-theta form]\label{prop:partial-theta}
For all $m,n\in\Z_+$,
\begin{align}
K_\rho(m,n) & = \frac{q^{[m(m+1)+n(n+1)]/4}}
{\sqrt{(q;q)_m(q;q)_n}\,(-\rho^{-1},-q\rho,q;q)_\infty}\nonumber\\
&\quad \times \sum_{r=0}^{m+n}D_r^{m,n}\rho^{-r+(m+n)/2}
\left[
\Theta_q\!\left(\rho^{-2}q^{2r-m-n-1}\right) +
\rho^{-1}\Theta_q\!\left(\rho^{-2}q^{2r-m-n+1}\right)
\right].
\label{eq:partial-theta-kernel}
\end{align}
\end{proposition}
\begin{proof}
Substitute \eqref{eq:Pn-cont-qinv-Hermite}, \eqref{eq:mu-l-explicit}, and the Laurent expansion above into \eqref{eq:spectral-series-kernel}.  For fixed $j,k$, with $r=j+k$, the $l$-dependent factor is
\[
q^{2l^2}(\rho^{-2}q^{2r-m-n-1})^l + \rho^{-1}q^{2l^2}(\rho^{-2}q^{2r-m-n+1})^l.
\]
Summing over $l\ge1$ gives the two partial theta functions.  The remaining finite double sum can then be grouped according to $r=j+k$, and the corresponding inner coefficient sum is precisely $D_r^{m,n}$.
\end{proof}

\section{Comparison with the \texorpdfstring{$q>1$}{q>1} regime}\label{sec:qgreater1-main}

For $0<q<1$, the coefficientwise Jacobi limit does not by itself specify a self-adjoint operator, and the finite ensemble supplies the missing boundary datum.  For $q>1$, the limiting moment problem is determinate, so the usual spectral-projection mechanism of Borodin--Olshanski~\cite{BO} applies.

Throughout this section let
\[
q>1, \qquad N = N_M \to \infty, \qquad M-N_M \to \infty, \qquad
\rho_M := p_M q^{2N_M} \longrightarrow \rho\in(0,\infty),
\]
and put
\[
Q := q^{-1}\in(0,1), \qquad r_M := p_M q^M = \rho_M q^{M-2N}.
\]

There is one sign change relative to Section~\ref{sec:finite-spectral} that must be fixed before taking the limit.  Write the two off-diagonal coefficients in the unsymmetrized difference equation as
\[
A_x := 1-q^{x-M}, \qquad C_x := -p_M(1-q^x).
\]
For $q>1$ both are positive in the interior, and the weight relation is $w_x A_x = w_{x+1}C_{x+1}$. Thus conjugation by the positive square root of the weight gives the positive off-diagonal $+\sqrt{A_x C_{x+1}}$, and reflection $n=M-x$ does not change its sign.  Accordingly, if $H_{p_M,M}$ denotes the reflected symmetrization for $q>1$, then
\begin{equation}\label{eq:qgreater1-ungauged-offdiag}
(H_{p_M,M})_{n,n+1} = +\sqrt{p_M(q^{-n-1}-1)(1-q^{M-n})}, \qquad 0\le n<M,
\end{equation}
while its diagonal is still given by \eqref{eq:finite-reflected-diag}.  In particular, the fixed minus sign in \eqref{eq:finite-reflected-offdiag} is specific to the $0<q<1$ convention and is not analytically continued across $q=1$.

The critical scaling is determined by comparing the diagonal and off-diagonal coefficients of the centered matrix.  Let
\[
C_M := \Lambda_N I-H_{p_M,M}.
\]
Using \eqref{eq:qk-finite-eigenvalue} and \eqref{eq:qgreater1-ungauged-offdiag},
\begin{align*}
(C_M)_{n,n+1} & = -\sqrt{r_M Q^n(1-Q^{n+1})(1-Q^{M-n})},\\
(C_M)_{n,n} & = (r_M-1)Q^n + (1-\rho_M)Q^N.
\end{align*}
If $M-2N\to-\infty$, then $r_M \to0$ and the unscaled centered operator has vanishing off-diagonal coefficients and diagonal limit $-Q^n$.  If $M-2N\to+\infty$, then after division by $r_M$ the off-diagonal coefficients again vanish and the diagonal limit is $Q^n$.  A scalar normalization can therefore keep the off-diagonal coefficient at nonzero order while keeping the diagonal coefficient finite only when $r_M \asymp1$; since $\rho_M \to \rho\in(0,\infty)$, this is equivalent to
\[
M-2N = O(1).
\]
This is the critical window for scalar normalization with the reflection point fixed.

Since $M-2N$ is bounded and integer-valued, we pass to a subsequence on which
\[
M-2N = \kappa
\]
for all sufficiently large $M$.  Then $r_M = \rho_M q^{M-2N} \longrightarrow \rho q^\kappa$, so the limiting Jacobi coefficients depend on $\kappa$.  The parity constraint is $M\equiv\kappa\pmod 2$ for all sufficiently large $M$.  Set
\[
c := (\rho q^\kappa)^{-1}>0.
\]
Define the normalized matrix on $\cH_M$ by
\begin{equation}\label{eq:LM-definition}
L_M := \frac{\Lambda_N I-H_{p_M,M}}{p_M q^M} = r_M^{-1}C_M.
\end{equation}
Its coefficients are therefore
\begin{align}
(L_M)_{n,n+1} & = -\sqrt{\frac{Q^n}{r_M}(1-Q^{n+1})(1-Q^{M-n})},\label{eq:qgreater1-L-offdiag}\\
(L_M)_{n,n} & = \left(1 - \frac1{r_M}\right)Q^n +
\frac{1-\rho_M}{r_M}Q^N.\label{eq:qgreater1-L-diag}
\end{align}
Let $\widehat L_M = L_M \oplus 0$ on $\cH$ according to Section~\ref{subsec:common-hilbert-space}.

To compare with the standard positive-Jacobi-gauge Al-Salam--Carlitz I recurrence, introduce the fixed alternating unitary
\begin{equation}\label{eq:qgreater1-gauge}
Ge_n := (-1)^ne_n, \qquad G_M := G|_{\cH_M},
\end{equation}
and set
\[
L_M^+ := G_M L_M G_M, \qquad \widehat L_M^+ := L_M^+ \oplus 0 = G\widehat L_M G.
\]
Thus $L_M^+$ has the same diagonal as $L_M$ and the opposite, positive, off-diagonal coefficients.

\subsection{The critical Jacobi limit}

\begin{proposition}[Positive-gauge Al-Salam--Carlitz I limit]\label{prop:asc-limit-main}
For every fixed $n\in\Z_+$, the Jacobi coefficients of $L_M^+$ converge to
\begin{equation}\label{eq:asc-limit-coefficients}
a_n^{\ASC} = \sqrt{cQ^n(1-Q^{n+1})}, \qquad b_n^{\ASC} = (1-c)Q^n.
\end{equation}
The limiting Jacobi recurrence is the Al-Salam--Carlitz I recurrence with base $Q$ and parameter $a_{\ASC}=-c$.  If $U_n^{(-c)}(x;Q)$ is the monic Al-Salam--Carlitz I polynomial in the convention of Koekoek--Lesky--Swarttouw, then the corresponding orthonormal polynomial is
\begin{equation}\label{eq:asc-orthonormal-polynomial}
P_n^{\ASC}(x) = \frac{U_n^{(-c)}(x;Q)}{\sqrt{c^n(Q;Q)_n Q^{n(n-1)/2}}}.
\end{equation}
In the special case $c=1$, equivalently $\rho q^\kappa=1$, these are the discrete $Q$-Hermite I polynomials.
\end{proposition}
\begin{proof}
Since $r_M = \rho_M q^{M-2N} \longrightarrow \rho q^\kappa = c^{-1}$ and $Q^N\to0$, equations \eqref{eq:qgreater1-L-offdiag}--\eqref{eq:qgreater1-L-diag} and the gauge \eqref{eq:qgreater1-gauge} give the limiting coefficients in \eqref{eq:asc-limit-coefficients}.

For Al-Salam--Carlitz I, KLS \cite[Eq.~(14.24.3)]{KLS} give the monic recurrence
\[
xU_n^{(a)}(x;Q) = U_{n+1}^{(a)}(x;Q) + (a+1)Q^nU_n^{(a)}(x;Q) -
aQ^{n-1}(1-Q^n)U_{n-1}^{(a)}(x;Q).
\]
Setting $a=-c$ and dividing by the square roots of the squared norms from the orthogonality relation \cite[Eq.~(14.24.2)]{KLS} gives exactly \eqref{eq:asc-limit-coefficients} and \eqref{eq:asc-orthonormal-polynomial}.  The specialization $a=-1$ to discrete $Q$-Hermite I is KLS \cite[Eq.~(14.24.12)]{KLS}.
\end{proof}

Let $L_0$ be the operator with coefficients \eqref{eq:asc-limit-coefficients} on $c_{00}(\Z_+)$.  Since $a_n^{\ASC} = \sqrt c\,Q^{n/2}\sqrt{1-Q^{n+1}}$, we have
\[
\sum_{n=0}^\infty \frac1{a_n^{\ASC}} = \infty.
\]
Carleman's criterion for Jacobi matrices therefore implies that $L_0$ is essentially self-adjoint; see, for example, Simon~\cite[Corollary~4.5]{Simon}.  We denote its unique self-adjoint closure by $L$.

\begin{theorem}[Strong resolvent convergence in the critical $q>1$ regime]\label{thm:qgreater1-sr-main}
Under the assumptions above,
\begin{equation}\label{eq:qgreater1-strong-resolvent}
\widehat L_M^+ \xrightarrow{\mathrm{s.r.}} L, \qquad \text{and hence} \qquad
\widehat L_M \xrightarrow{\mathrm{s.r.}} GLG
\end{equation}
on the fixed Hilbert space $\cH$.
\end{theorem}
\begin{proof}
For every $f\in c_{00}(\Z_+)$, Proposition~\ref{prop:asc-limit-main} gives $\widehat L_M^+f \longrightarrow L_0 f$. Since $L_0$ is essentially self-adjoint, $c_{00}(\Z_+)$ is a core for $L$.  The common-core criterion \cite[Theorem~VIII.25(a)]{RS} therefore yields $\widehat L_M^+ \xrightarrow{\mathrm{s.r.}} L$. Conjugating by the fixed unitary $G$ gives the second convergence.
\end{proof}

\subsection{Negative spectral projections and the alternating gauge}

For $q>1$, the ordering of the finite $q$-Krawtchouk eigenvalues is reversed: $\Lambda_{k+1}-\Lambda_k = (1-q)(q^{-k-1}+p_M q^k)<0$. Consequently $\Lambda_N-\Lambda_k<0$ for $k<N$ and $\Lambda_N-\Lambda_k\ge0$ for $k\ge N$. Since the ensemble occupies the degrees $k=0,\ldots,N-1$, its original reflected correlation kernel, viewed as an operator and extended by zero to $\cH$, is
\begin{equation}\label{eq:finite-negative-projection}
\widehat K_{M,N} := K_{M,N} \oplus 0 = \ind_{(-\infty,0)}(\widehat L_M).
\end{equation}
Again the artificial zero eigenspace is excluded by the open half-line, so this extension does not alter the finite kernel on $\cH_M$.

Conjugating the entire projection by the same alternating gauge, define
\begin{equation}\label{eq:finite-negative-projection-gauged}
\widehat K_{M,N}^+ := G\widehat K_{M,N}G = \ind_{(-\infty,0)}(\widehat L_M^+).
\end{equation}
On the finite block this is $K_{M,N}^+ := G_M K_{M,N}G_M$, so entrywise
\begin{equation}\label{eq:finite-kernel-gauge}
K_{M,N}^+(m,n) = (-1)^{m+n}K_{M,N}(m,n).
\end{equation}

The Al-Salam--Carlitz I orthogonality measure for parameter $-c<0$ is supported on two geometric lattices; see KLS~\cite[Eq.~(14.24.2)]{KLS}.  In our normalization the $e_0$-spectral measure is
\begin{equation}\label{eq:asc-spectral-measure}
\mu = \sum_{j=0}^\infty\mu_j^+\delta_{Q^j} + \sum_{j=0}^\infty\mu_j^-\delta_{-cQ^j},
\end{equation}
where
\begin{equation*}
\mu_j^+ = \frac{Q^j}{(-c;Q)_\infty(Q;Q)_j(-Q/c;Q)_j},
\end{equation*}
\begin{equation}\label{eq:asc-negative-mass}
\mu_j^- = \frac{cQ^j}{(1+c)(-Q/c;Q)_\infty(Q;Q)_j(-cQ;Q)_j},
\end{equation}
and $\sum_{j\ge0}\mu_j^+ + \sum_{j\ge0}\mu_j^- = 1$. Because $e_0$ is cyclic for the Jacobi operator $L$, the spectrum is the closure of the support of its scalar $e_0$-spectral measure.  Hence
\begin{equation}\label{eq:asc-spectrum}
\sigma(L) = \{0\} \cup \{Q^j:j\ge0\} \cup \{-cQ^j:j\ge0\}.
\end{equation}
The point $0$ is the common accumulation point of the two branches, but the scalar measure \eqref{eq:asc-spectral-measure} has no atom there.  Hence
\begin{equation}\label{eq:asc-no-zero-atom}
E_L(\{0\}) = 0.
\end{equation}

\begin{theorem}[The $q>1$ limiting kernel]\label{thm:qgreater1-kernel-main}
Under the assumptions of Theorem~\ref{thm:qgreater1-sr-main},
\begin{equation}\label{eq:qgreater1-projection-convergence}
\ind_{(-\infty,0)}(\widehat L_M^+) \longrightarrow \ind_{(-\infty,0)}(L)
\end{equation}
strongly on $\cH$.  Equivalently,
\[
\widehat K_{M,N} \longrightarrow G\ind_{(-\infty,0)}(L)G
\]
strongly.  Define the standard positive-gauge Al-Salam--Carlitz kernel by
\begin{equation}\label{eq:qgreater1-spectral-series}
K_{\ASC}(m,n) := \sum_{j=0}^\infty \mu_j^-P_m^{\ASC}(-cQ^j)P_n^{\ASC}(-cQ^j).
\end{equation}
Then, for fixed $m,n\in\Z_+$,
\begin{equation}\label{eq:qgreater1-original-kernel-limit}
K_{M,N}(m,n) \longrightarrow (-1)^{m+n}K_{\ASC}(m,n).
\end{equation}
\end{theorem}
\begin{proof}
By \eqref{eq:asc-no-zero-atom}, $0$ is not an eigenvalue of $L$.  From the finite eigenvalue ordering, the degree-$k=0$ eigenvalue of $\widehat L_M^+$ converges to $-c$; hence, for any fixed $a<-c-1$, both $\widehat L_M^+$ for all sufficiently large $M$ and $L$ have no spectrum below $a$.  Therefore the spectral-projection convergence theorem \cite[Theorem~VIII.24(b)]{RS}, together with Theorem~\ref{thm:qgreater1-sr-main}, gives $\ind_{(-\infty,0)}(\widehat L_M^+) \longrightarrow \ind_{(-\infty,0)}(L)$ strongly.  The spectral expansion defining $K_{\ASC}$ is the negative branch of the Al-Salam--Carlitz I measure above, and the checkerboard identity then converts the positive-gauge matrix elements back to the original reflected kernel.
\end{proof}

\begin{remark}[Gauge invariance of the determinantal process]\label{rem:qgreater1-gauge-invariance}
For any finite set of points $x_1,\ldots,x_r$, conjugating the kernel matrix by the diagonal matrix $\diag ((-1)^{x_1},\ldots,(-1)^{x_r})$ leaves its determinant unchanged.  Hence $K_{M,N}$ and $K_{M,N}^+$ define the same finite determinantal point process, and the same is true for their limiting kernels.
\end{remark}

\begin{remark}[Finite spectral picture]
The two branches in \eqref{eq:asc-spectrum} can already be read from the finite eigenvalues.  For fixed $k$,
\[
\frac{\Lambda_N-\Lambda_k}{p_M q^M} \longrightarrow -cQ^k,
\]
whereas for fixed $j$ and $k=M-j$,
\[
\frac{\Lambda_N-\Lambda_{M-j}}{p_M q^M} \longrightarrow Q^j.
\]
If instead $k_M \to\infty$ and $M-k_M \to\infty$, then
\[
\frac{\Lambda_N-\Lambda_{k_M}}{p_M q^M} \longrightarrow 0.
\]
Thus the two extreme degree ranges converge to the two geometric Al-Salam--Carlitz I branches, while the central degrees collapse toward their common accumulation point $0$.
\end{remark}

\subsection{Integrable form}

The positive-gauge spectral series \eqref{eq:qgreater1-spectral-series} admits a Christoffel--Darboux representation similar to that in the $0<q<1$ regime.

On the lattice $x_j := -cQ^j$, $j\in\Z_+$, the Al-Salam--Carlitz I $Q$-difference equation, KLS~\cite[Eq.~(14.24.5)]{KLS}, gives a second-order difference equation in $j$ with eigenvalue $Q^{-n}$.  In a convenient normalization its off-diagonal coefficients are
\begin{equation}\label{eq:asc-dual-coefficients}
A_j = -\frac1{cQ^{2j+1}}, \qquad C_j = -\frac{(1+cQ^j)(1-Q^j)}{cQ^{2j}}, \qquad C_0 = 0,
\end{equation}
with the diagonal coefficient determined by the eigenvalue equation.  The negative-branch masses satisfy detailed balance $\mu_j^-A_j = \mu_{j+1}^-C_{j+1}$. Consequently the weighted Green identity applied to the two sequences $j \longmapsto P_m^{\ASC}(x_j)$ and $j \longmapsto P_n^{\ASC}(x_j)$ telescopes.  Since $C_0=0$, summing over $j=0,\ldots,J$ gives
\begin{align*}
&(Q^{-m}-Q^{-n})\sum_{j=0}^J\mu_j^-
P_m^{\ASC}(x_j)P_n^{\ASC}(x_j)\\
& \qquad = -\mu_J^-A_J
\bigl(
P_m^{\ASC}(x_J)P_n^{\ASC}(x_{J+1}) -
P_m^{\ASC}(x_{J+1})P_n^{\ASC}(x_J)
\bigr).
\end{align*}
At the boundary $j=J$, $x_J=-cQ^J\to0$.  Write
\[
\mathcal A_n := P_n^{\ASC}(0), \qquad \mathcal B_n := (P_n^{\ASC})'(0).
\]
Since $P_n^{\ASC}(x) = \mathcal A_n+x\mathcal B_n + O(x^2)$, one obtains
\begin{align*}
&P_m^{\ASC}(x_J)P_n^{\ASC}(x_{J+1}) -
P_m^{\ASC}(x_{J+1})P_n^{\ASC}(x_J)\\
& \qquad = c(1-Q)Q^J
(\mathcal A_m \mathcal B_n - \mathcal B_m \mathcal A_n) +
O(Q^{2J}).
\end{align*}
Moreover, by \eqref{eq:asc-negative-mass} and \eqref{eq:asc-dual-coefficients},
\[
-\mu_J^-A_J = \frac{Q^{-J-1}}{(1+c)(-Q/c;Q)_\infty(Q;Q)_J(-cQ;Q)_J}.
\]
Hence the boundary term at $j=J$ converges to
\[
\frac{c(1-Q)}{Q(1+c)(Q;Q)_\infty(-cQ;Q)_\infty(-Q/c;Q)_\infty}
(\mathcal A_m \mathcal B_n - \mathcal B_m \mathcal A_n).
\]
Passing to $J\to\infty$ in the partial spectral sum gives the following formula.

\begin{proposition}[Integrable form of the $q>1$ kernel]\label{prop:qgreater1-integrable-main}
Define
\begin{equation}\label{eq:qgreater1-C-constant}
C_{c,Q} := \frac{c(1-Q)}{Q(1+c)(Q;Q)_\infty(-cQ;Q)_\infty(-Q/c;Q)_\infty}.
\end{equation}
For $m\ne n$,
\begin{equation}\label{eq:qgreater1-wronskian}
K_{M,N}(m,n) \longrightarrow (-1)^{m+n}K_{\ASC}(m,n) = (-1)^{m+n}C_{c,Q} \frac{
\mathcal A_m \mathcal B_n - \mathcal B_m \mathcal A_n}{Q^{-m}-Q^{-n}}.
\end{equation}
For $m=n$ we retain the spectral series \eqref{eq:qgreater1-spectral-series}.
\end{proposition}

The forward-shift identity KLS~\cite[Eq.~(14.24.7)]{KLS}, after the normalization \eqref{eq:asc-orthonormal-polynomial}, gives
\begin{equation*}
\mathcal B_n = \frac{\sqrt{1-Q^n}}{(1-Q)\sqrt c}\,
Q^{-(n-1)/2}\mathcal A_{n-1}
\quad(n\ge1),\qquad \mathcal B_0=0.
\end{equation*}
Thus \eqref{eq:qgreater1-wronskian} may, if desired, be written entirely in terms of the values $\{P_k^{\ASC}(0)\}_{k\ge0}$.

\subsection*{Acknowledgements}
The author is grateful to Anton Nazarov for useful discussions. This work was supported by the Beijing Natural Science Foundation Grant No. IS26011.

\end{document}